\documentclass[12pt]{amsart}
\usepackage{amssymb}
\usepackage{amscd}
\usepackage{color}
\usepackage{url}
\usepackage{hyperref}
\usepackage{comment}
\usepackage{booktabs}
\usepackage{mwe}

\numberwithin{equation}{section}

\theoremstyle{plain}
\newtheorem{theorem}[subsection]{Theorem}
\newtheorem{proposition}[subsection]{Proposition}
\newtheorem{lemma}[subsection]{Lemma}
\newtheorem{corollary}[subsection]{Corollary}

\theoremstyle{definition}

\newtheorem{definition}[subsection]{Definition}

\theoremstyle{remark}

\newcommand{\pe}[1]{\langle#1\rangle}
\newcommand{\norm}[1]{\left\|#1\right\|}

\newcommand{\C}{\mathbb{C}}

\newcommand{\N}{\mathbb{N}}
\newcommand{\D}{\mathbb{D}}
\newcommand{\T}{\mathbb{T}}

\newcommand{\X}{\mathcal{X}}

\newcommand{\R}{\mathbb{R}}

\newcommand{\Fi}{\varphi}

\DeclareMathOperator{\Log}{Log}
\DeclareMathOperator{\PR}{Re}

\newcommand{\HH}{\mathcal{H}}

\newcommand{\Span}{\textnormal{span}}

  {\par \hfill \fbox{}}

\makeatletter
\def\author@andify{%
	\nxandlist {\unskip ,\penalty-1 \space\ignorespaces}%
	{\unskip {} \@@and~}%
	{\unskip \penalty-2 \space \@@and~}%
}
\makeatother

\title[Dunford property for composition operators on $H^p$-spaces]{Dunford property for composition \\ operators on $H^p$-spaces}

\author{E. A. Gallardo-Guti\'errez}
\address{Departamento de An\'alisis Matem\'atico y Matem\'atica Aplicada, Facultad de Matem\'aticas, Universidad Complutense de Madrid, Plaza de Ciencias 3, 28040 Madrid, Spain and Instituto de Ciencias Matem\'aticas ICMAT (CSIC-UAM-UC3M-UCM), Madrid, Spain.}
\email{eva.gallardo@mat.ucm.es}

\author{F. J. Gonz\'alez-Do\~na}
\address{Departamento de Matem\'aticas, Escuela Polit\'ecnica Superior, Universidad Carlos III de Madrid, Avda. de la Universidad 30, 28911 Legan\'es (Madrid), Spain.}
\email{fragonza@math.uc3m.es}

\author{M. Monsalve-L\'opez}
\address{Departamento de An\'alisis Matem\'atico y Matem\'atica Aplicada, Facultad de Matem\'aticas, Universidad Complutense de Madrid, Plaza de Ciencias 3, 28040 Madrid, Spain.}
\email{migmonsa@ucm.es}

\thanks{Authors are partially supported by Plan Nacional  I+D grants no. PID2019-105979GB-I00 and PID2022-137294NB-I00(Spain), the Spanish Ministry of Science and Innovation, through the ``Severo Ochoa Programme for Centres of Excellence in R\&D'' (CEX2019-000904-S) and from the Spanish National Research Council, through the ``Ayuda extraordinaria a Centros de Excelencia Severo Ochoa'' (20205CEX001).}

\date{July 18th, 2023, Revised: April 1st, 2024}

\begin{document}

\begin{abstract}
The Dunford property $(C)$ for composition operators on $H^p$-spaces ($1<p<\infty$), as well as for their adjoints, is completely characterized within the class of those induced by linear fractional transformations of the unit disc. As a consequence, it is shown that the Dunford property is stable in such a class addressing a particular instance of a question posed by Laursen and Neumann.
\end{abstract}

\maketitle

\section{Introduction and Preliminaries}

In 1959, E. Bishop used a Banach space version of the analytic dua\-lity principle established by  Silva, K\"{o}the, Grothendieck and others operator theorists to study connections between spectral decomposition properties of a Banach space operator and its adjoint. Of particular interest in this setting are operators satisfying the \emph{Bishop property} $(\beta)$ since E. Albrecht and J. Eschmeier \cite{AS} developed a complete duality theory for them. Actually, the property $(\beta)$ turns out to characterize restrictions of decomposable operators to closed invariant subspaces, and Albrecht and Eschmeier's analytic functional model shows that every Banach space operator is similar to the quotient of an operator with the Bishop property. Similarly, quotients of decomposable operators are determined by the decomposition property $(\delta)$ and in fact, both properties $(\beta)$ and $(\delta)$ are completely dual: an operator has one if and only if its adjoint has the other.

\smallskip

The class of decomposable operators, introduced by Foia\c{s} \cite{Foias} in the sixties, constitutes  a generalization of spectral operators (in the sense of Dunford) and many operators in Hilbert spaces as unitary operators, self-adjoint operators or more generally, normal operators are decomposable. Recall that a linear bounded operator $T$ acting on a Banach space $X$ is \emph{decomposable} if for every open cover $\{U_1, U_2\}$ of the complex plane $\mathbb{C}$, there exist closed  $T$-invariant  subspaces $X_1$ and $X_2$ of $X$ so that $X= X_1 + X_2$ and the spectrum of the restriction $T|_{X_i}$ is contained in $U_i$, for $i= 1,2$.
It turns out that $T$ is decomposable if and only if $T$ satisfies both properties $(\beta)$ and $(\delta)$.

\smallskip
Although the class of operators which are decomposable is not stable by restricting them to closed invariant subspaces, it turns out that the class of those which satisfy the \emph{Dunford property} $(C)$ is. This property,  pioneered by Dunford in the sixties, has played a major role in the development of the theory of spectral operators. Being the Dunford property $(C)$ a weaker property than the Bishop one (see the Millers' example \cite{Miller-Miller 98}), one has that an operator  $T$ is decomposable if and only if it satisfies both properties $(C)$ and $(\delta)$. We refer to the monograph by Laursen and Neumann \cite{LN} for more on this subject.

\smallskip

Studying local spectral properties for classes of \emph{concrete operators} acting on function spaces leads naturally to interesting questions from the function theoretic perspective, and this is one of the driving aim of the present work. At this regard, it is worthy to point out the works of the Millers' and coauthors \cite{BMM} or \cite{MMS} or the one by Aleman and Persson \cite{AP} regarding the classical Ces\`aro operator or Ces\`aro type operators.

\smallskip

One of the most significant classes of operators having the property $(C)$ is the one consisting of \emph{multipliers} acting on  \emph{functional Banach spaces} (see \cite{CMc} for the definition and \cite[Proposition 1.6.9]{LN} for the result). In particular, it turns out that every multiplier of the Hardy spaces $H^p(\D)$, $1\leq p<\infty$ has the property $(C)$ but fails property $(\delta)$. In this work, we characterize the Dunford property $(C)$ for composition operators $C_{\varphi}$, and their adjoints $C_{\varphi}^*$, whenever they are induced by linear fractional transformations $\varphi$ of the unit disc $\mathbb{D}$ in $H^p$-spaces ($1<p<\infty$). They constitute a first approach to a more comprehensive  analysis of the study of composition operators induced by any holomorphic self map of $\D$. At this regard, it is worthy to remark that while the spectrum of composition operators induced by linear fractional transformations of  $\mathbb{D}$ in $H^p$-spaces is known (see \cite[Chapter 8]{CMc} and the references therein), determining the spectrum of general composition operators in $H^p(\D)$ remains as an open (and difficult) question (see \cite{GM} for recent results in this setting).

\smallskip

As we will show, the Dunford property $(C)$ is strongly dependent on the fixed-point configuration of the inducing symbol $\varphi$, and fails drastically in those cases when the operator does not have it, that is,  the operator does not have even the \emph{single-valued extension property}. Moreover, when both $C_{\varphi}$ and $C_{\varphi}^*$ have the  Dunford property $(C)$, it turns out that they are decomposable operators. In Table  \ref{tabla} we summarise the characterization, and before proceeding further with comments, we recall some preliminaries.

\smallskip

\subsection*{The setting} For $1 \leq p < \infty$, the \emph{Hardy space} $H^p(\mathbb{D})$ is defined as the complex Banach space of analytic functions $f : \mathbb{D}\to \mathbb{C}$ with finite norm given by
	\[
	\|f\|_{H^p(\mathbb{D})} := \lim_{r\to 1^-} \bigg(\frac{1}{2\pi} \int_0^{2\pi} |f(re^{i\theta})|^p \, d\theta \bigg)^{1/p}.
	\]
	For $p=\infty$, the space $H^\infty(\mathbb{D})$ consists of all bounded analytic function on $\mathbb{D}$ equipped with the supremum norm.
	
	\medskip
	
	Fixed $1 < p < \infty$, it is a well-known fact that each element of the dual space $H^p(\mathbb{D})^*$ may be regarded univocally as a continuous functional of the form
	\[
	\varphi_h(f) := \frac{1}{2\pi} \int_0^{2\pi} f(e^{i\theta})\,h(e^{i\theta})\, d\theta,
	\]
	for some function $h \in L^{p'}(\mathbb{T})/H_0^{p'}(\mathbb{D})$, where $\mathbb{T}$ denotes the unit circle,  $\frac{1}{p} + \frac{1}{p'} = 1$ and
	\[
	H_0^{p'}(\mathbb{D}):=\bigg\{g \in H^{p'}(\mathbb{D}) : \int_0^{2\pi} g(e^{i\theta}) \, d\theta = 0\bigg\}=zH^{p'}(\mathbb{D}).
	\]
	Equivalently, any functional $\varphi : H^p(\mathbb{D}) \to \mathbb{C}$ in the dual space $H^p(\mathbb{D})^*$ can be identified with a function $g \in H^{p'}(\mathbb{D})$ such that
	\[
	\varphi(f) := \frac{1}{2\pi} \int_0^{2\pi} f(e^{i\theta})\,\overline{g(e^{i\theta})}\, d\theta.
	\]
	In fact, this correspondence establishes a sesquilinear dual pairing between $H^p(\mathbb{D})^*$ and $H^{p'}(\mathbb{D})$. In particular, the Hardy space $H^p(\mathbb{D})$ is reflexive for every $1 < p <\infty$.
	
	
\subsection*{Composition operators.} If $\varphi : \mathbb{D}\to\mathbb{D}$ is a holomorphic map, the \emph{composition operator} $C_\varphi$ is defined as
	\begin{align*}
		C_\varphi : H^p(\mathbb{D}) & \to H^p(\mathbb{D}) \\
		f & \mapsto f\circ\varphi.
	\end{align*}
	The boundedness of each composition operator $C_\varphi : H^p(\mathbb{D})  \to H^p(\mathbb{D})$ is guaranteed by the \emph{Littlewood Subordination Theorem}. Clearly, $C_\varphi$ is invertible in $H^p(\mathbb{D})$ if and only if the function $\varphi$ is an \emph{automorphism} of the unit disk $\mathbb{D}$.
	
	\smallskip
	
In this work, we consider composition operators induced by \textit{linear fractional transformations}  of $\D$, namely, holomorphic maps of $\D$ given by
	$$ \Fi(z) = \frac{az+b}{cz+d}, \quad ad-bc \neq 0$$
where
$$ |b\overline{d}-a\overline{c}|+|ad-bc| \leq |b|^2-|d|^2 $$
(see, for instance, \cite{CDMV}). The linear fractional transformations can be classified according to their fixed points in the Riemann sphere $\widehat{\C}:= \C\cup \{ \infty\}$.
A linear fractional map $\Fi$ that takes $\D$ into itself is called \textit{parabolic} if it has just one fixed point, which must lie in $\T$. We will distinguish the cases of \textit{parabolic automorphisms} (PA) and \textit{parabolic non-automorphisms} (PNA).
	
	\medskip
	
	The other possible situation is when $\Fi$ has two fixed points, and the classification depends on the location of such fixed points. One of them must lie in $\overline{\D}.$ If it lies on $\T$, then $C_\Fi$ is called \textit{hyperbolic}. If the other fixed point lies in $\T$, then $\Fi$ is an automorphism and will be called \textit{hyperbolic automorphism} (HA). If it lies in $\widehat{\C}\setminus \overline{\D}$, $\Fi$ is a \textit{hyperbolic non-automorphism of first kind} (HNA I), and if it lies in $\D$ it will be called a \textit{hyperbolic non-automorphism of second kind} (HNA II).
	
	\medskip
	
	Finally, if $\Fi$ has no fixed points in $\T$, one of them must lie in $\D$, and the other fixed point must lie in $\widehat{\C}\setminus \overline{\D}.$ Here, we distinguish two situations: if $\Fi$ is an automorphism it is called \textit{elliptic} (EA), and otherwise, $\Fi$ is said to be \textit{loxodromic} (LOX).
	
	\medskip
	
	After conjugations with appropriate linear fractional transformations, every of these maps can be expressed in a standard form as described in Table  \ref{tabla}. This classification is well-known in the literature and has been used extensively (see, for instance, \cite{GR, LLPR}).

\subsection*{Local spectral theory} Local spectral theory arose as an effort to extend some of the most important properties of normal operators to broader classes of operators in the context of Banach spaces. In this regard, the notion of \emph{decomposable operators} emerges:  bounded linear operators $T$ on a Banach space $X$ for which each splitting of the spectrum $\sigma(T)$ corresponds to a sum decomposition of the space $X$ consisting of $T$-invariant subspaces. The class of decomposable operators is considerably large: for instance, all operators admitting a sufficiently rich functional calculus are decomposable, as well as any operator with totally disconnected spectrum.
Among the local spectral properties, of particular relevance is the \textit{single-valued extension property}.
	
	\begin{definition}
A linear bounded operator $T$ in $X$ has the single-valued extension property (abbreviated as SVEP) if for every open subset $U\subset \C$ and every analytic function $f: U \rightarrow X$ satisfying
		\[
		(T-zI)f(z) = 0, \qquad \text{for all }z \in U	
		\]
		it follows that $f \equiv 0.$
	\end{definition}
	
	It is clear that every operator  whose point spectrum has empty interior automatically possesses the SVEP. This property is closely related to the concept of \textit{local resolvent} of an operator. Given a linear bounded operator $T$ in $X$  and a vector $x\in X$, the local resolvent of $T$ at $x$ (denoted by $\rho_T(x)$) is the union of all the open sets $U\subset \C$ for which there exists an analytic function $f_x : U \rightarrow X$ satisfying
	\begin{equation}\label{ecuacion resolvente local}
		(T-zI)f_x(z) = x, \qquad z \in U.
	\end{equation}
	
For those operators $ T$ enjoying the SVEP, the functional equation \eqref{ecuacion resolvente local} has a unique solution for each $x \in X$. Hence, in that case, there exists a unique local resolvent function $f_x : \rho_T(x) \to X$ satisfying \eqref{ecuacion resolvente local}.
	
	\smallskip
	
	The \textit{local spectrum} of $T$ at a vector $x\in X$ is defined as $\sigma_T(x) := \C \setminus \rho_T(x)$. With this notion, one may define the \textit{local spectral subspaces} (or simply the \textit{spectral subspaces}) of $T $ associated to a subset $F\subset \C$ as
	\[
	X_T(F) := \{ x\in X \,  :  \, \sigma_T(x) \subseteq F \}.	
	\]
	For each $F\subset \C$ the spectral subspace $X_T(F)$ is a (non-necessarily closed) linear manifold in $X$ which is hyperinvariant by the operator $T$.
With the definition of (local) spectral subspaces at hands, we recall the \textit{Dunford property} $(C)$:
	\begin{definition}
A linear bounded operator $T$ in $X$ has the Dunford's property $(C)$ if the local spectral subspace $X_T(F)$ is closed for every closed set $F\subset \C.$
	\end{definition}
	It turns out that every decomposable operator has property $(C)$, and every operator with property $(C)$ enjoys the SVEP (see \cite[Theorem 1.2.7 and Proposition 1.2.19]{LN}).
	
	\smallskip
	
	As it was pointed out in the introduction, in this work we provide a complete characterization of the Dunford's property $(C)$ for composition operators $C_\Fi$ induced by linear fractional maps in $H^p(\D)$, $(1<p<\infty)$,  as well as  for their adjoints $C_\Fi^*.$ Next table indicates, in terms of the configuration of the fixed points of the symbol $\Fi$, whether the operators $C_\Fi$ or $C_\Fi^*$ have the property $(C)$.
	
	\medskip

	\begin{table}[h!]
		\centering
		\begin{tabular}{cclc}
			\toprule
			Symbol & Fixed points & Canonical form of $\Fi$ & Property $(C)$\\ \midrule
			HA & $1,\, -1$ & $\Fi(z) = \frac{r+z}{1+rz}, \quad 0<r<1$ & $C_\Fi^*$ (Thm. \ref{propiedad (C) automorfismo hiperbolico})   \\[0.5em]
			EA & $0,\, \infty$ & $\Fi(z) = \omega z, \quad \omega \in \T$ & $C_\Fi$ and $C_\Fi^*$ (\cite{Smith})  \\[0.5em]
			PA & $1$ & $\Fi(z) = \frac{(2-a)z+a}{-az+2+a}, \quad \PR(a)= 0$ &  $C_\Fi$ and $C_\Fi^*$ (\cite{Smith}) \\[0.5em]
			HNA I & $1,\, \infty$ & $\Fi(z) =rz+1-r, \quad 0<r<1$ & $C_\Fi^*$ (Thm. \ref{teorema hiperbolico no-automorfismo}) \\[0.5em]
			HNA II & $0,\, 1$ & $\Fi(z) = \frac{rz}{1-(1-r)z}, \quad 0<r<1$ & $C_\Fi$ (Thm. \ref{teorema hiperbolico no-automorfismo ii})  \\[0.5em]
			PNA & $1$ & $\Fi(z) = \frac{(2-a)z+a}{-az+2+a},\quad \PR(a)> 0$ &  $C_\Fi$ and $C_\Fi^*$ (\cite{Shapiro}) \\[0.5em]
			LOX & $c\in \D,\, \infty$ & $\Fi(z) = a(z-c)+c, \quad |a|+|1-a||c|\leq 1 $ &  $C_\Fi$ and $C_\Fi^*$ (\cite{Kamowitz})  \\[0.5em]\bottomrule
		\end{tabular}
\medskip
\medskip
\caption{Dunford property $(C)$ for composition operators and their adjoints acting on $H^p(\D)$ for $1<p<\infty$}\label{tabla}
	\end{table}

	It is worth mentioning that the characterization of the Dunford's property $(C)$ is independent of $1 < p < \infty$ and relies upon exclusively on the nature of $\varphi$. Likewise, in the cases in which  $\Fi$ is an elliptic automorphism, a parabolic self-map of $\D$  or a loxodromic one,  the induced composition operators are either known to be decomposable in $H^p(\D)$ or the result is straightforward: if $\Fi$ is an elliptic or parabolic automorphism, $C_\Fi$ is generalized scalar \cite[Theorems 1.1 and 1.2]{Smith}, and therefore decomposable. If $\Fi$ is a parabolic non-automorphism, Shapiro showed that $C_\Fi$ is decomposable in \cite{Shapiro}. Finally, if $\Fi$ is loxodromic, the spectrum of $C_\Fi$ is totally disconnected \cite{Kamowitz}, so $C_\Fi$ is decomposable as well. As a consequence, the adjoints of such operators are also decomposable \cite[Theorem 2.5.3]{LN}, and both $C_\Fi$ and $C_{\Fi}^*$ acting on $H^p(\D)$ and $H^p(\D)^*$ respectively for $1<p<\infty$ have the Dunford property $(C)$.

\smallskip	

In order to deal with the remaining cases, in Section  \ref{seccion 2}	we prove a sufficient condition for linear bounded operators acting on Banach spaces ensuring the Dunford  property $(C)$. In Section \ref{seccion 3}, we consider invertible composition operators and show that if  $\Fi$ is a hyperbolic automorphism of $\D$,
$C_\Fi$  does not have the SVEP for $1\leq p < \infty.$ In other words, all the invertible composition operators in $H^p(\mathbb{D})$ are decomposable except the hyperbolic ones. Nevertheless, in this case, we prove that $C_\Fi^*$ has the Dunford property $(C)$ for every $1<p<\infty.$  As a consequence, we are able to describe all the local spectral subspaces and characterize the local spectra. In Section \ref{seccion 4}, we consider the non-invertible composition operators,  dealing with the cases of hyperbolic non-automorphisms of first and second kind. In particular, as a consequence, we describe the local spectra of such operators. Finally, as a byproduct of the results summarized in Table \ref{tabla}, we show that the class of composition operators induced by linear fractional transformations is stable under the Dunford property $(C)$.

\section{A sufficient condition}\label{seccion 2}
	In this section, we present a general result for linear bounded operators acting on  Banach spaces which provides a sufficient condition in order to have the Dunford property $(C)$, along with some more (local) spectral features. This result, which can be understood as an extension of \cite[Proposition 1.6.12]{LN}, will be applied in the context of composition operators.
	
	\smallskip
	
	We start by recalling the \textit{glocal spectral subspaces} associated to an operator $T$ in order to deal with a variant of (local) spectral subspaces which are better suited for operators which do not have the SVEP. Given a closed set $F \subseteq \C$, the  glocal spectral subspace $\mathcal{X}_T(F)$ consists of all $x \in X$ for which there exists an analytic function $f : \C\setminus F\rightarrow X$ such that
	\[
	(T-zI)f(z) = x, \qquad z \in \C\setminus F.
	\]
	It is clear that the equality $\X_T(F) = X_T(F)$ holds for every closed set $F\subset \C$ whenever the operator $T$ enjoys the SVEP. But in general, we have the inclusion $\X_T(F)\subseteq X_T(F)$.
	
	\medskip
	
Since we are dealing with Hardy spaces $H^p(\D)$, for every operator $T$ and closed set $F\subset \C$, we will denote by $H^p_T(F)$ and $\HH^p_T(F)$ the associated local spectral subspaces and glocal spectral subspaces, respectively.
	
	 \medskip
	
	Glocal spectral subspaces behave well with respect to the adjoint operation. Indeed, by \cite[Proposition 2.5.1]{LN}, if $T : X  \to X$ is a linear bounded operator and $F$ and $G$ are two disjoint closed subsets of $\C$, then
	\begin{equation}\label{glocal containments}
		\X_T(F) \subseteq {^\perp\X^*_{T^*}(G)} \qquad\text{and}\qquad \X^*_{T^*}(F) \subseteq \X_T(G)^\perp,
	\end{equation}
	where $M^\perp$ denotes the annihilator of a linear manifold $M \subseteq X$, while ${^\perp N}$ denotes the preannihilator of a linear manifold $N\subseteq X^*$. Actually, since $H^p(\D)$ is reflexive for $1<p<\infty$, having in mind the dual pairing between $H^p(\D)$ and $H^{p'}(\D)$, ($\frac{1}{p} + \frac{1}{p'}=1$), the property \eqref{glocal containments} reads as
	\[
	\mathcal{H}^p_T(F)\subseteq \mathcal{H}^{p'}_{T^*}(G^*)^\perp,	
	\]
	where $A^* := \{\overline{z}  :  z \in A\}$ for every subset $A\subseteq \C$.
	
	\smallskip
	
	Before stating the main result of the section, let us recall that for each $x \in X$, the \textit{local spectral radius} of $T$ at $x$ is the quantity
$$r_T(x) := \limsup\limits_{n\rightarrow\infty} \norm{T^nx}^{1/n}.$$
If $T$ has the SVEP, then $r_T(x) = \max \{|\lambda| : \lambda \in \sigma_T(x)\}$ for every non-zero $x\in X.$ Besides, we denote by $r(T)$ the spectral radius of $T$ and by $\eta(\sigma(T))$ the full spectrum of $T$, namely, the union of $\sigma(T)$ and every bounded connected component of $\rho(T)$.
	
	\begin{theorem}\label{teorema propiedad C}
		Let $T :X \to X$ be a bounded linear operator in a complex Banach space $X$. Assume further that for each non-empty relatively open subset $U \subseteq \sigma(T)$, the glocal spectral subspace $\X_T(\overline{U})$ is dense in $X$. Then, $\X^*_{T^*}(F) = \{0\}$ for every closed subset $F\subsetneq \sigma(T^*)$. As a consequence, the following properties hold:
		\begin{enumerate}
			\item [(i)] If $\sigma(T^*)$ is not a singleton, then $\sigma_\text{\normalfont{p}}(T^*) = \emptyset$.
			\item [(ii)] $T^*$ has the Dunford's property $(C)$.
			\item [(iii)] $\sigma_{T^*}(x) = \sigma(T^*)$ for every $x \in X^*\setminus \{0\}.$
			\item [(iv)] $r_{T^*}(x) = r(T^*)$ for every $x \in X^*\setminus \{0\}.$
			\item [(v)] If $M$ is any non-trivial closed invariant subspace for $T^*$, then
			\[
			\sigma(T^*)\subseteq\sigma(T^*|_M) \subseteq \eta(\sigma(T^*)).
			\]
		\end{enumerate}
	\end{theorem}
	
	\begin{proof}
		Let $F \subsetneq \sigma(T^*)$ be a closed subset and consider $U$ a non-empty relatively open set in $\sigma(T^*)$ such that $F\cap \overline{U} = \emptyset.$ Now, by \eqref{glocal containments}, it follows that $$\X^*_{T^*}(F) \subseteq \X_T(\overline{U})^\perp = \{0\},$$ where the last equality holds by the density of $\X_T(\overline{U})$.
		
		\smallskip
		
		First, let us prove property \normalfont{(i)}. Assume that $\sigma(T^*)$ has at least two points and take $\lambda \in \sigma(T^*)$. As a consequence,
		\[
		\ker(T^*-\lambda I)\subseteq \X^*_{T^*}(\{ \lambda \}) = \{0\},
		\]
		so $\sigma_\text{\normalfont{p}}(T^*)= \emptyset.$ In particular, $T^*$ has the SVEP and $\X^*_{T^*}(F) = X^*_{T^*}(F)$ for every closed set $F\subseteq \sigma(T^*)$.
		
		\medskip
		
		In order to prove properties \normalfont{(ii)-(v)}, first observe that $T^*$ also enjoys the SVEP when $\sigma(T^*)$ is reduced to a singleton, because in that case, its point spectrum has empty interior. To show \normalfont{(ii)}, let $F\subseteq \sigma(T^*)$ be an arbitrary closed set. Then,
		\[
		X^*_{T^*}(F) = \begin{cases} \{0\}, & \text{ if } F\subsetneq \sigma(T^*),\\ X^*, & \text{ if } F = \sigma(T^*),\end{cases}
		\]
		and so $T^*$ has the Dunford property $(C)$.
		
		\smallskip
		
		Now, both properties \normalfont{(iii)} and \normalfont{(iv)}  follow immediately from property \normalfont{(ii)}. Finally, \normalfont{(v)} is a direct application of \cite[Proposition 1.2.16 (e)]{LN} and \cite[Theorem 0.8]{RR}.
	\end{proof}

	\section{Dunford property for invertible composition operators}\label{seccion 3}
	
	In this section, we consider the Dunford property $(C)$ for invertible composition operators $C_\varphi$ in $H^p(\mathbb{D})$ and their adjoints $C_\varphi^*$. As it was pointed out previously,  the Dunford property $(C)$ will be independent of $1 < p < \infty$ and rely upon exclusively on the nature of the automorphism $\varphi$.

	\smallskip

	In 1996, Smith proved that if $\Fi$ is an elliptic or parabolic automorphism, $C_\Fi$ is generalized scalar (see \cite[Theorems 1.1 and 1.2]{Smith}), and therefore a decomposable operator. Accordingly, we are left with composition operators induced by  hyperbolic automorphisms.

	\begin{theorem}\label{propiedad (C) automorfismo hiperbolico}
		Let $\Fi$ be a hyperbolic automorphism of the unit disc $\D$. Then, $C_\Fi$ does not have the SVEP in $H^p(\mathbb{D})$ for any $1\leq p < \infty$ and $C_\Fi^*$ has the Dunford's property $(C)$ for every $1<p<\infty$.
	\end{theorem}

	In order to prove such a result, our first step is to understand the set of eigenvectors of  $C_\varphi$ when $\varphi$ is a hyperbolic automorphism.	In this case, the linear fractional model establishes similarity between $C_\varphi$ and a composition operator $C_{\varphi_r}$ with hyperbolic symbol $\varphi_r$ of the canonical form:
	\[
	\varphi_r(z) := \frac{z+r}{1+rz}, \quad 0  <r < 1,
	\]
	where $\varphi_r$ fixes the points $\pm 1$, being $1$ the Denjoy-Wolff point. In this light, the spectrum of $C_{\varphi_r}$ in $H^p(\mathbb{D})$,  was characterized by Nordgren \cite{Nordgren}:
	\[
	\sigma(C_{\varphi_r}|_{H^p(\D)}) = \bigg\{\lambda \in \mathbb{C} \, : \,  \bigg(\frac{1+r}{1-r}\bigg)^{-1/p} \leq |\lambda| \leq \bigg(\frac{1+r}{1-r}\bigg)^{1/p}\bigg\}
	\]
	and its point spectrum is $\sigma_\text{\normalfont{p}}(C_{\varphi_r}|_{H^p(\D)}) = \mathrm{int}\big(\sigma(C_{\varphi_r}|_{H^p(\D)})\big)$. In fact, for each eigenvalue $\left( \frac{1+r}{1-r}\right)^\lambda$ ($-1/p<\PR(\lambda)<1/p$) of $C_{\varphi_r}$, the associated eigenspace is infinite-dimensional, and an example of associated eigenvector is given by
	\[
	w_\lambda (z) := \bigg(\frac{1+z}{1-z}\bigg)^\lambda, \quad\text{ where }\; -1/p < \mathrm{Re}(\lambda) < 1/p.
	\]
	
It is clear that the adjoint $C_\varphi^*$ of any composition operator with hyperbolic automorphism symbol $\varphi$ is similar to $C_{\varphi_r}^*$ for some $0 < r < 1$ and $\sigma(C_{\varphi_r}^*|_{H^p(\D)^*}) = \sigma(C_{\varphi_r}|_{H^p(\D)})$.
	
\smallskip
	
In order to apply Theorem \ref{teorema propiedad C} in this context, we prove that, in general, some subsets of eigenfunctions of $C_\varphi$ feature strong spanning properties in each Hardy space $H^p(\mathbb{D})$ with $1 < p < \infty$ (see \cite[Lemma 1]{CG1} in the Hilbert space setting).

	\begin{proposition}\label{densidad_hiperbolico}
		Let $\varphi : \mathbb{D} \to \mathbb{D}$ be a hyperbolic automorphism and consider the composition operator $C_\varphi : H^p(\mathbb{D}) \to H^p(\mathbb{D})$ for some $1 < p < \infty$. Then, for every set $F \subseteq \sigma_\text{\normalfont{p}}(C_\varphi)$ with a cluster point inside $\sigma_\text{\normalfont{p}}(C_\varphi)$, the linear manifold
		\[
		\mathrm{span} \big\{\!\ker(C_\varphi - \lambda I) : \lambda \in F\big\}
		\]
		is dense in $H^p(\mathbb{D})$.
	\end{proposition}
	
	\begin{proof}
		Without loss of generality, suppose that $\varphi = \varphi_r$ for some $0 < r< 1$. Denote by $B_{1/p}$ the open vertical strip $\{\lambda \in \mathbb{C} \, : \, -1/p < \mathrm{Re}(\lambda) < 1/p\}$ and consider the following $H^p(\mathbb{D})$-valued function $\varrho : B_{1/p} \to H^p(\mathbb{D})$, defined by $\varrho(\lambda) := w_\lambda$ where, as above,
		\[
		w_\lambda (z) := \bigg(\frac{1+z}{1-z}\bigg)^\lambda, \quad\text{ for each }\; -1/p < \mathrm{Re}(\lambda) < 1/p.
		\]
		Since $C_{\varphi_r} w_\lambda = \big(\tfrac{1+r}{1-r}\big)^\lambda w_\lambda$, observe that the mapping $\lambda \mapsto \big(\tfrac{1+r}{1-r}\big)^\lambda$ sends $B_{1/p}$ onto the point spectrum $\sigma_\text{\normalfont{p}}(C_{\varphi_r})$. Hence, there is a correspondence between $F$ and a set $\Lambda \subseteq B_{1/p}$ which, by continuity, has a cluster point inside $B_{1/p}$.
		
		\smallskip
		
		Now, we claim that $\varrho : B_{1/p} \to H^p(\mathbb{D})$ is an analytic function. To do so, consider $p'>1$ such that $\frac{1}{p}+\frac{1}{p'}=1$. It is enough to show that, for each $g \in H^{p'}(\D),$ the map $$ \lambda \in B_{1/p} \mapsto \pe{\varrho(\lambda),g}= \frac{1}{2\pi} \int_0^{2\pi} \left( \frac{1+e^{i\theta}}{1-e^{i\theta}} \right)^\lambda \overline{g(e^{i\theta})} d\theta $$ is holomorphic, which follows immediately by means of standard methods involving Morera's Theorem.
		
		\medskip

		\medskip
		
		At this point, consider a bounded linear functional $\eta \in H^p(\mathbb{D})^*$ such that $\langle w_\lambda, \eta \rangle = 0$ for every $\lambda \in \Lambda$. The map $\lambda \mapsto \langle w_\lambda, \eta \rangle$ is a holomorphic function on $B_{1/p}$ which annihilates on the set $\Lambda$. Taking into account that $\Lambda$ has a cluster point, we conclude that $\langle w_\lambda, \eta \rangle = 0$ for every $\lambda \in B_{1/p}$. Accordingly, our quest reduces to prove that $\mathrm{span}\{w_\lambda : \lambda \in B_{1/p}\}$ is dense in $H^p(\mathbb{D})$. To do so, we will show that
		\begin{equation}\label{density 1}
			\overline{\mathrm{span}\{w_{it} : t \in \mathbb{R}\}}^{H^p(\mathbb{D})} = H^p(\mathbb{D})
		\end{equation}
		for every $1 \leq p < \infty$. In \cite[Lemma 1]{CG1}, the authors proved the case $p=2$ by means of the Paley-Wiener Theorem, which clearly yields $1\leq p\leq 2$.
		
		\medskip
		
		Fix the principal branch of the logarithm, and consider the band
		\[
		\Pi := \{z\in \mathbb{C} : \;  -\pi/2 < \mathrm{Im}(z) <\pi/2 \}.
		\]
		Note that proving \eqref{density 1} is equivalent to prove that
		\begin{equation}\label{density 2}
			\overline{\mathrm{span}\big\{e^{\lambda z} : \,  z \in \Pi ,\;  \mathrm{Re}(\lambda)=0 \}}^{H^p(\Pi)}=H^p(\Pi),
		\end{equation}
		where the space $H^p(\Pi)$ consists of all holomorphic functions $f : \Pi \to \mathbb{C}$ for which there exists a harmonic function $u$ such that $|f(z)|^p\leq u(z)$ for all $z\in \Pi$ (we refer to Duren's book \cite[Ch. 10]{Du} for more on these spaces).

\smallskip

The assertion in \eqref{density 2} is, indeed, a particular instance of a more general situation studied in \cite{BGY} by Bracci, Gallardo-Gutiérrez and Yakubovich, where the authors give a complete characterization of those non-elliptic semigroups of holomorphic self-maps of the unit disc for which the linear span of the eigenfunctions of the generator of the corresponding semigroup of composition operators is weak-star dense in $H^\infty$ of the Koenigs domain $\Omega$ associated to the semigroup. Likewise, the authors give necessary and sufficient conditions for completeness in $H^p(\Omega)$.

\smallskip

We borrow the proof of the instance \eqref{density 2} from \cite{BGY} and include it for the sake of completeness, referring to the manuscript for more details and further results.

\medskip

Write $\lambda=it$ with $t\in \R$ and observe that $e^{itz} \in H^{\infty}(\Pi)$ for every $t\in \R$. Hence, $\|e^{itz}\|_{H^p(\Pi)}\leq C$ for some absolute constant $C > 0$. Accordingly, the integral
		\begin{equation*}\label{int-convergencia}
			\int_0^{\infty} \|e^{itz}\|_{H^p(\Pi)} e^{-t\beta} dt
		\end{equation*}
		converges uniformly on $\Pi$ for every $\beta \in \mathbb{C}$ with $\mathrm{Re}(\beta) > 0$. Given $\varepsilon>0$, take $M>0$ such that
		\[
		\int_M^{\infty}\|e^{itz}\|_{H^p(\Pi)} e^{-t\beta}\, dt<\varepsilon.
		\]
		
		Note that for those $\beta \in \mathbb{C}$ with $\mathrm{Re} (\beta) >- \pi/2 $ the integral
		\begin{equation}\label{eq1-laplace}
		\int_0^\infty  e^{it z} e^{-t \beta}\, dt= \frac{1}{-iz+\beta},
		\end{equation}
		for $z\in \Pi$.  Analogously, for those $\beta \in \mathbb{C}$ with $\mathrm{Re} (\beta) < \pi/2 $ the integral
		\begin{equation}\label{eq2-laplace}
		\int_{-\infty}^0  e^{it z} e^{-t \beta}\, dt= \frac{1}{iz-\beta},
		\end{equation}
		for $z\in \Pi$.

Denote by $F_{\beta}$ the function in \eqref{eq1-laplace}, namely $F_{\beta}(z)= \frac{1}{-iz+\beta}$ for $z\in \Pi$. Clearly, $F_{\beta}\in H^{\infty}(\Pi)$ and it is not difficult to prove that each $I_\beta^M(z) =\int_0^M e^{itz} e^{-t\beta}\, dt$ converges to $F_{\beta}$ in $H^{\infty}(\Pi)$ as $M\to \infty$. Note that $I_\beta^M$ can be seen as the Laplace transform of the complex Borel measure compactly supported in  $[0,+\infty)$ defined by $\mu=\chi_{[0,M]}(t)\, dt$. Approximating $\mu$ (in the sense of measures)  by finite linear combinations of measures $\mu_n$ given by
\begin{equation}\label{eq:approx-measure}
	\mu_n=\sum_{j=1}^{n}  \mu\bigg(\bigg[\frac {(j-1)M} n,\frac {jM} n\bigg)\bigg)\delta_{\frac{jM}{n}},
\end{equation}
it follows that the functions
\[
	P_n(z)=\int_0^{+\infty} e^{itz}d\mu_n(t), \qquad (z\in \Pi)
\]
belong to $\hbox{span}\{e^{itz}: t\geq 0\}$ for all $n \in \N$, and moreover, there exists $C>0$ such that $\sup_{z\in\Pi}|P_n(z)|\leq C$. Hence, for every $z\in \Pi$ fixed,
\[
	\lim_{n\to\infty} \big|P_n(z)-I_\beta^M(z) \big| = \lim_{n\to\infty}\left|\int_0^M e^{itz} \,d\mu_n(t)- \int_0^M e^{itz}e^{-t\beta} \,dt \right|=0,
\]
where the last limit is $0$  because the latter integral converges to $0$ as $n\to \infty$ for $z$ fixed since $\{\mu_n\}$ converges to $\mu$ in the sense of measures.  Accordingly, $\{P_n\}$ converges weak-star in $H^\infty(\Pi)$ to $I_\beta^M$.

\smallskip

Arguing similarly with \eqref{eq2-laplace}, it follows that the rational functions with simple poles outside $\overline{\Pi}$ are contained in the closure of $\hbox{span}\{e^{itz}: t\in \R\}$  in the weak-star topology of $H^\infty(\Pi)$. Now, each rational function in $H^\infty(\Pi)$ can be uniformly approximated in $\Pi$ by rational functions with simple poles. In addition, it is well-known that rational functions with poles outside $\overline{\Pi}$ are weak-star dense in $H^\infty(\Pi)$ (see, for instance, \cite[Corollary~1]{Do0}). Accordingly, $\hbox{span}\{e^{it z}: t\in \R\}$ is weak-star dense in $H^\infty(\Pi)$ and therefore, weak-star dense in $H^p(\Pi)$. Since $1<p<\infty$, it is  dense in $H^p(\Pi)$ in the weak topology and Mazur's Lemma (see \cite[Corollary 3, Chapter 2]{Diestel}) ensures the density in $H^p(\Pi)$.
	\end{proof}

As previously mentioned, in \cite{CG1} the authors established that the linear span of the eigenfunctions of any invertible composition operator of hyperbolic symbol is dense in $H^2(\mathbb{D})$. Using that fact, and reminding that for each $1 < p \leq 2$ the inclusion $H^2(\mathbb{D})\hookrightarrow H^p(\mathbb{D})$ is continuous and that $H^2(\mathbb{D})$ lies densely within $H^p(\mathbb{D})$, Proposition \ref{densidad_hiperbolico} follows directly for $1 < p < 2$ (and, by the same token, for $p=1$).

\smallskip

As a consequence of Proposition \ref{densidad_hiperbolico}, we obtain the density of certain glocal spectral subspaces associated to $C_\Fi$, which will allow us to apply Theorem \ref{teorema propiedad C} in order to obtain the desired result.

	\begin{proposition}\label{lema glocal densos automorfismo}
		Let $\varphi$ be a hyperbolic automorphism of $\D$ and $C_\varphi$ the induced composition operator in $H^p(\mathbb{D})$ for some $1 < p < \infty$. Then, the glocal spectral subspace $\mathcal{H}^p_{C_\varphi}(\overline{U})$ is dense in $H^p(\mathbb{D})$ for every non-empty relatively open subset $U$ of $\sigma(C_\varphi)$.
	\end{proposition}
	
	\begin{proof}
		Let $U$ be a non-empty relatively open subset of $\sigma(C_\varphi)$. Then,
		\[
		\ker(C_\varphi - \lambda I) \subseteq \mathcal{H}^p_{C_\varphi}(\{\lambda\}) \subseteq \mathcal{H}^p_{C_\varphi}(\overline{U}), \quad \text{for all } \lambda \in \overline{U}.
		\]
		Taking into account that $\mathcal{H}^p_{C_\varphi}(\overline{U})$ is a linear manifold, we conclude that
		\[
		\mathrm{span}\{\ker(C_\varphi - \lambda I) : \lambda \in \overline{U} \} \subseteq \mathcal{H}^p_{C_\varphi}(\overline{U}).
		\]
		Finally, since the closure of a non-empty open set always has cluster points, our result follows directly from Proposition \ref{densidad_hiperbolico}.
	\end{proof}

With both propositions at hands, the proof of Theorem \ref{propiedad (C) automorfismo hiperbolico} is almost straightforward.

\begin{proof}[Proof of Theorem \ref{propiedad (C) automorfismo hiperbolico}]
	As mentioned above, for each $1<p<\infty$, the point spectrum of $C_\Fi$ in $H^p(\D)$ is an annulus and an argument involving the eigenfunctions yields that $C_{\varphi}$ has not the SVEP (see \cite[Theorem 1.4]{Smith}). On the other hand,  Theorem \ref{teorema propiedad C} and Proposition  \ref{lema glocal densos automorfismo} imply that $C_\Fi^*$ has the Dunford's property $(C)$ in $H^p(\D)^*$ for every $1<p<\infty$.
\end{proof}
	
As a consequence,  we are able to describe all the local spectral manifolds and characterize some of the (local) spectral properties of $C_\Fi:$
	
	\begin{corollary}\label{corolario hiperbolico automorfismo}
		Let $\varphi$ be a hyperbolic automorphism of $\mathbb{D}$ and $C_\varphi$ the induced composition operator in $H^p(\mathbb{D})$ for some  $1 < p < \infty$. Let $C_\varphi^*$ be its adjoint operator in $H^p(\mathbb{D})^*$ and $p' > 1$ such that $\tfrac{1}{p} + \tfrac{1}{p'} = 1$. Then:
		\begin{enumerate}
			\item [(i)] The point spectrum of $C_\Fi^*$ is empty.
			\item [(ii)]   $H^{p'}_{C_\Fi^*}(F) = \{0\}$ for every closed set $F\subsetneq \sigma(C_{\Fi_r}^*)$.
			\item [(iii)] $\sigma_{C_\varphi^*}(f) = \sigma(C_\varphi^*)$ for every non-zero $f \in H^p(\mathbb{D})^*$.
			\item [(iv)] $r_{C_\varphi^*}(f) = r(C_\varphi^*)$ for every non-zero $f \in H^p(\mathbb{D})^*$.
			\item [(v)] If $M$ is any non-trivial closed invariant subspace for $C_{\varphi}^*$, then
			\[
			\sigma(C_\varphi^*|_M) = \sigma(C_\varphi^*) \quad \text{ or } \quad \sigma(C_\varphi^*|_M) = \big\{\lambda \in \mathbb{C} : |\lambda|\leq r(C_\varphi^*)\big\}.
			\]
		\end{enumerate}
	\end{corollary}

\begin{proof} Properties (i)-(iv) follow as a byproduct of Theorem \ref{teorema propiedad C} and Proposition  \ref{lema glocal densos automorfismo}. A little extra argument is needed to state (v). Observe that
$$\sigma(C_\Fi^*)\subseteq \sigma({C_\Fi^*}|_M) \subseteq \eta(\sigma(C_\Fi^*)).$$
Now, a result of Scroggs \cite{Scroggs} states that if a point $\lambda$ of a hole $K$ of $\sigma(C_\Fi^*)$ belongs to $\sigma({C_\Fi^*}|_M)$, then $K\subset \sigma({C_\Fi}|_M),$ which yields (v).
\end{proof}

\section{Dunford property for non-invertible composition operators}\label{seccion 4}
	
In this section, we characterize in $H^p(\D)$ the Dunford property $(C)$ for non-invertible composition operators $C_\Fi$ induced by linear fractional maps as well as their adjoints. Since loxodromic maps induce composition operators with disconnected spectrum in $H^p(\D)$ (and hence, they are decomposable) and the parabolic maps induce decomposable composition operators \cite{Shapiro}, we are left  with the hyperbolic non-automorphisms cases (of first and second kind).
	
\subsection{Hyperbolic non-automorphisms I}
Let $\Fi$ be a hyperbolic non-automorphism of first kind of $\D$, namely, a linear fractional map with a fixed point in $\T$ and the other one in $\widehat{\C}\setminus \overline{\D}.$

\begin{theorem}\label{teorema hiperbolico no-automorfismo}
Let $\Fi$ be a hyperbolic non-automorphism of first kind of $\D$. Then, $C_\Fi$ does not have the SVEP in $H^p(\mathbb{D})$ for any $1\leq p < \infty$ and $C_\Fi^*$ has the Dunford's property $(C)$.
\end{theorem}

If $\Fi$ is a hyperbolic non-automorphism of first kind of $\D$, then $\Fi$ is conjugated to one of its canonical form:
\begin{equation}\label{modelo hiperbolico no-automorfismo}
		\Fi_r(z) = rz+(1-r), \qquad 0<r<1.
	\end{equation}
Accordingly, in order to prove Theorem \ref{teorema hiperbolico no-automorfismo}, it suffices to consider the transformations described in \eqref{modelo hiperbolico no-automorfismo}. It is straightforward that the functions
\[
	e_s(z) = (1-z)^s, \qquad \PR(s) > - 1/p
\]
belong to $H^p(\D)$ and satisfy
\[
	C_{\Fi_r}e_s = r^se_s.
\]
Thus, $\{r^s : \PR(s) > -1/p\} = D(0, r^{-1/p}) \subseteq \sigma_\text{\normalfont{p}}(C_{\Fi_r}).$ Indeed, Kamowitz \cite{Kamowitz} showed that
\[
	\sigma(C_{\Fi_r})= \overline{D(0, r^{-1/p})}.
\]

Now, we proceed with the proof of Theorem \ref{teorema hiperbolico no-automorfismo}.

\begin{proof}[Proof of Theorem \ref{teorema hiperbolico no-automorfismo}]
	We may assume without loss of generality that $\Fi = \Fi_r$ for some $0<r<1$ given by \eqref{modelo hiperbolico no-automorfismo}. Note that the map
	\[
		s \in \big\{s \in \C : \PR(s)>-1/p \big\} \longmapsto e_s \in H^p(\D)
	\]
	is a vector-valued analytic map and consider $V = \{ z \in D(0,r^{-1/p}): \PR(z) > 0\}$. The map $f:V \rightarrow H^p(\D)$ given by
	\[
		f(\lambda) = e_{\frac{\Log(\lambda)}{\Log(r)}},
	\]
	where $\Log:\C\setminus (-\infty,0] \rightarrow \C$ is the principal branch of the logarithm, is also a vector-valued analytic map that satisfies
	\[
		(C_{\Fi_r}-\lambda)f(\lambda) = 0, \qquad \text{for every }\lambda\in V,
	\]
	which proves that $C_{\Fi_r}$ does not have the SVEP, as claimed.
	
	\smallskip
	
	Now, we will apply Theorem \ref{teorema propiedad C} to derive the Dunford's property $(C)$ for the adjoint $C_{\Fi_r}^*$. First, let us show that $\HH^{p}_{C_{\Fi_r}}(\overline{U})$ is dense in $H^{p}(\D)$ for every non-empty relatively open subset $U$ of $\sigma(C_{\Fi_r})$.
	
	\smallskip
	
	Since $\HH^{p}_{C_{\Fi_r}}(\overline{U})$ contains $\ker (C_{\Fi_r}-\lambda)$ for every $\lambda \in \overline{U}$, it follows that
	\[
		\Span \{e_s : r^s  \in \overline{U} \}\subseteq \HH^{p}_{C_{\Fi_r}}(\overline{U}).
	\]
	
	Now, consider $\eta \in \Span \{e_s : r^s \in \overline{U} \}^\perp \subset H^{p}(\D)^*.$ Then, the holomorphic map
	\[
		s \in \big\{ s \in \C: \PR(s)  > -1/p\big\} \longmapsto \langle e_s,\eta \rangle\in \C
	\]
	vanishes in every $s$ such that $r^s \in \overline{U}.$ Since $\overline{U}$ has accumulation points, it follows that $\langle e_s,\eta\rangle = 0$ for every $s \in \C$ with $\PR(s) > -1/p$. Now, observe that
	\[
		\Span\{e_n : n \in \N \cup\{0\}\} \subseteq \Span \big\{e_s : \PR(s)> -1/p \big\}.
	\]
	But $\Span\{e_n : n \in \N \cup\{0\}\}$ contains all the polynomials, hence it follows that $f\equiv 0$ and so $\HH_{C_{\Fi_r}}^{p}(\overline{U})$ is dense as claimed. In conclusion,
	Theorem \ref{teorema propiedad C} yields that $C_{\Fi_r}^*$ has the Dunford property in $H^p(\D)^*$.
\end{proof}

As a consequence, we deduce the following:
	
	\begin{corollary}\label{corolario hiperbolico no-automorfismo}
		Let $\Fi:\D \rightarrow \D$ be a hyperbolic non-automorphism map of first kind and $C_{\Fi}$ the induced composition operator in $H^p(\D)$ for some $1< p < \infty$. Let $C_\varphi^*$ be its adjoint operator in $H^p(\D)^*$ and $p'>1$ such that $\tfrac{1}{p} + \tfrac{1}{p'} = 1$. Then:
		\begin{enumerate}
			\item [(i)] The point spectrum of $C_\Fi^*$ is empty.
			\item [(ii)] $H^{p'}_{C_\Fi^*}(F) = \{0\}$ for every closed set $F\subsetneq \sigma(C_{\Fi_r}^*)$.
			\item [(iii)] $\sigma_{C_\Fi^*}(f) = \sigma(C_\Fi^*)$ for every non-zero $f \in H^{p}(\D)^*$.
			\item [(iv)] $r_{C_\Fi^*}(f) = r(C_\Fi^*)$ for every non-zero $f \in H^{p}(\D)^*$.
			\item [(v)] If $M$ is a non-trivial closed invariant subspace for $C_{\Fi}^*$, then $\sigma(C_\Fi^*|_M) =\sigma(C_\Fi^*).$
		\end{enumerate}
	\end{corollary}

\subsection{Hyperbolic non-automorphisms II}
	Now, we consider $\Fi: \D\rightarrow \D$  hyperbolic non-automorphisms of second kind of $\D$, namely, linear fractional maps with a fixed point in $\D$ and the other one lying in $\T$.
	
	\begin{theorem}\label{teorema hiperbolico no-automorfismo ii}
		Let $\Fi : \D\rightarrow \D$ be a hyperbolic non-automorphism of second kind. Then, $C_\Fi$ has the Dunford's property $(C)$ in $H^p(\D)$ for every $1<p<\infty$. Moreover, $C_\Fi^*: H^p(\D)^* \to H^p(\D)^*$ does not have the SVEP for any $1<p<\infty$.
	\end{theorem}

	In order to prove this result, we consider the following maps, which have $0$ and $1$ as fixed points:
	\begin{equation}\label{modelo hiperbolico no-automorfismo ii}
		\Fi_r(z) = \frac{rz}{1-(1-r)z}, \qquad 0<r<1.
	\end{equation}
	It turns out that every hyperbolic non-automorphism of second kind is conjugated to one of the maps of the form \eqref{modelo hiperbolico no-automorfismo ii}. The spectrum of these composition operators $C_{\Fi_r}$ in $H^p(\D)$, for each $1\leq p < \infty$, was characterized by Kamowitz. Specifically,
	\[
		\sigma(C_{\Fi_r}) = \overline{D(0,r^{1/p})}\cup \{1\}.
	\]
	
	Our approach is inspired by a construction of Shapiro that appears in \cite[p. 864]{Shapiro2} and traces back to a joint paper with Bourdon \cite{BS}. The author shows that the operator $C_{\Fi_r}$ acting on the Hilbert space $H^2(\D)$ is unitarily equivalent to $I_1\oplus (rC_{\psi_r})^*,$  where $I_1$ is the identity acting on the one-dimensional space of constants, $\psi_r(z) = rz+(1-r)$ is a hyperbolic non-automorphism of first kind and $(rC_{\psi_r})^*$ acts in $H^2(\D)$.

\begin{lemma}\label{lema semejantes}
		Let $\Fi_r : \D \to \D$ be a hyperbolic non-automorphism of second kind of the form \eqref{modelo hiperbolico no-automorfismo ii}. Then, for every $1 < p < \infty,$ the operator $C_{\Fi_r}$  in $H^p(\D)$ is similar to the operator $I_1 \oplus (rC_{\psi_r})^*$, where $\psi_r(z) = rz+(1-r)$ and $(rC_{\psi_r})^*$ acts in $H^p(\D)$.
\end{lemma}

\begin{proof}
		Let $1<{p'}<\infty$ be such that $\frac{1}{p}+\frac{1}{{p'}}= 1$, and consider $C_{\Fi_r}^*$ in $H^{p'}(\D)$.  Let $H_0^{p'}(\D) := zH^{p'}(\D),$ that is, the closed subspace of $H^{p'}(\D)$-functions vanishing at $0$. It is easy to check that each functional in $H_0^{p'}(\D)^*$ can be identified with a function of $H_0^p(\D)$, with the same dual pairing that relates $H^{p'}(\D)^*$ and $H^p(\D)$.
		
		\smallskip
		
		Since $1$ is the reproducing kernel at $0$ and $\Fi_r(0) =0$, it follows that the rank-one operator $1\otimes 1$ in $H^{p'}(\D)$ commutes with $C_{\Fi_r}^*.$ Observe that $1\otimes 1$ is a projection onto the one-dimensional space consisting of constants $\pe{1}$, and so $I-1\otimes 1$ is a projection onto $zH^{p'} (\D).$ Hence, $C_{\Fi_r}^*$ also commutes with $I-1\otimes 1$ and therefore $zH^{p'}(\D)$ is invariant under $C_{\Fi_r}^*$.

		\smallskip
		
		Now, denote by $M_z : H^{p'}(\D) \rightarrow H_0^{p'}(\D)$ and $M_{1/z}: H_0^{p'}(\D)\rightarrow H^{p'}(\D)$ the bounded operators of multiplication by $z$ and $1/z$, respectively. Observe that these operators are inverses of each other. Similar arguments as in \cite[p. 35-36]{BS} yield that
		\[
			\langle C_{\Fi_r}^*f, z^n\rangle = \langle M_z(rC_{\psi_r})M_{1/z}f,z^n \rangle
		\]
		for every $f \in H_0^{p'}(\D)$ and $n\in\N$. Since polynomials are dense in $H_0^p(\D)$, it follows that
		\[
			C_{\Fi_r}^*|_{H_0^{p'}(\D)} = M_z\big(rC_{\psi_r}|_{H^{p'}(\D)}\big)M_{1/z}.	
		\]
		Finally, since $H^{p'}(\D) = \pe{1}\oplus H_0^{p'}(\D),$ and $C_{\Fi_r}^*|_{\pe{1}} = I_1$, we deduce that $C_{\Fi_r}^*$ is similar to $I_1\oplus (rC_{\psi_r})$, and the result follows  by taking  adjoints.
\end{proof}
	
With this result at hands, we are in position to prove Theorem \ref{teorema hiperbolico no-automorfismo ii}.
	
\begin{proof}[Proof of Theorem \ref{teorema hiperbolico no-automorfismo ii}]
	Without loss of generality, we may assume that $\Fi = \Fi_r$ for some $0 < r < 1$ of the canonical form in \eqref{modelo hiperbolico no-automorfismo ii}. By Lemma \ref{lema semejantes}, using the decomposition $H^p(\D) = \langle 1 \rangle \oplus H_0^p(\D)$, the operator $C_{\Fi_r}$ is similar to $T:=I_1\oplus (rC_{\psi_r})^*$. Thus, by Corollary \ref{corolario hiperbolico no-automorfismo}, it follows that $\sigma_\text{\normalfont{p}}(C_{\Fi_r})= \{1\}.$ As a consequence, both $C_{\Fi_r}$ and $T$ have the SVEP.
	
	\smallskip
	
	Now, let $F\subsetneq \sigma(C_{\Fi_r})$ be a closed set. Consider $f =f_1\oplus f_2 \in H^p_T(F)$ and $\Gamma = \Gamma_1\oplus\Gamma_2 : \C\setminus F \rightarrow \pe{1}\oplus H_0^p(\D)$ satisfying
	\[
		(T-zI)\Gamma(z) = f \qquad\text{for every } z \in \C\setminus F.
	\]
	Observe that $(I_1-zI)\Gamma_1(z) = f_1$ and $((rC_{\psi_r})^*-zI)\Gamma_2(z) = f_2$ for every $z \in \C\setminus F$, so $H^p_T(F) = H^p_{I_1}(F)\oplus H^p_{rC_{\psi_r}^*}(F).$ Finally, by Theorem \ref{teorema hiperbolico no-automorfismo}, $H^p_{rC_{\psi_r}^*}(F)$ is closed, so $H^p_{C_{\Fi_r}}(F)$ is closed as well and $C_{\varphi_r}$ has the Dunford property $(C)$, as claimed.
	
	\smallskip
	
	It remains to show that $C_{\varphi_r}^*$  does not have the SVEP in $H^p(\D)^*$. Just recall that $C_{\varphi_r}^*$ is similar to $I_1\oplus (rC_{\psi_r})$ acting in $H^{p'}(\D)$, and it is straightforward to check that the fact that $C_{\psi_r}$ does not enjoy the SVEP (see Theorem \ref{teorema hiperbolico no-automorfismo}) implies that $C_{\Fi_r}^*$ does not either, which yields the statement.
\end{proof}
	
	Indeed, the identity $H^p_T(F) = H^p_{I_1}(F)\oplus H^p_{rC_{\psi_r}^*}(F)$ obtained previously gives us plenty of information about local spectral features of $C_{\Fi_r}$, which we recollect in our next result:
	
	\begin{corollary}
		Let $\Fi_r : \D \to \D$ be a hyperbolic non-automorphism of second kind of the form \eqref{modelo hiperbolico no-automorfismo ii} and $1<p<\infty.$ Consider $C_{\Fi_r}$ acting on $H^p(\D)$ and $F\subsetneq \sigma(C_{\Fi_r})$  a closed set. Then:
		\begin{enumerate}
			\item[(i)] $\sigma_\text{\normalfont{p}}(C_{\varphi_r})=\{1\}$ and its associated eigenspace is $\langle 1 \rangle$.
			\item[(ii)] $\displaystyle H^p_{C_{\Fi_r}}(F) = \begin{cases} \{0\}, & \text{if } F\subsetneq \overline{D(0,r^{1/p})}, \\ H_0^p(\D),  & \text{if } F= \overline{D(0,r^{1/p})}, \\ \langle 1\rangle, & \text{if } F=\{1\}.\end{cases}$
			\item[(iii)] If $f$ is a non-zero function in $H^p(\mathbb{D})$, then
			\[
				\sigma_{C_{\Fi_r}}(f) = \begin{cases} \{1\}, & \text{if } f \in \langle 1\rangle, \\ \overline{D(0,r^{1/p})},  & \text{if } f \in H_0^p(\mathbb{D}), \\ \sigma(C_{\varphi_r}), & \text{otherwise}.\end{cases}
			\]
			\item[(iv)] If $f$ is a non-zero function in $H^p(\mathbb{D})$, then
			\[
				r_{C_{\Fi_r}}(f) = \begin{cases} r^{1/p},  & \text{if } f \in H_0^p(\mathbb{D}), \\ 1, & \text{otherwise}.\end{cases}
			\]
			\item[(v)] Let $M\subset H^p(\D)$ be a non-trivial closed invariant subspace for $C_{\Fi_r}$. Then,
			\[
				\sigma(C_{\Fi_r}|_M)\in \big\{ \{1\}, \overline{D(0,r^{1/p})}, \sigma(C_{\Fi_r})\big\}.
			\]
		\end{enumerate}
		\end{corollary}

\medskip	

\begin{proof}
Note that (i) has been actually proved in the proof of Theorem \ref{teorema hiperbolico no-automorfismo ii}. To prove (ii), recall that $C_{\Fi_r}$ is similar to $T:= I_1\oplus (rC_{\psi_r})^*$ (where $(rC_{\psi_r})^*$ is acting on $H^p(\D)$) and $H^p_T(F) = H^p_{I_1}(F)\oplus H^p_{rC_{\psi_r}^*}(F)$ for every closed set $F\subsetneq \sigma(C_{\Fi_r}).$  By Corollary \ref{corolario hiperbolico no-automorfismo}, if $F\subsetneq \overline{D(0,r^{1/p})}$, then $H^p_{rC_{\psi_r}^*}(F) = \{0\}$. Since $H^p_{I_1}(F) = \{0\}$ as well, we deduce that $H^p_{C_{\Fi_r}}(F) = \{0\}$ in that case.

\smallskip

On the other hand, if $F=\overline{D(0,r^{1/p})}$, then $H^p_{rC_{\psi_r}^*}(F)=H^p(\D)$, and bringing it back via the isomorphism $M_{z} :H^p(\D)\rightarrow H_0^p(\D)$ we deduce that $H^p_{C_{\Fi_r}}(F) = H_0^p(\D),$ as stated. Finally, it is straightforward to check that $H^p_T({{1}}) = \langle 1 \rangle,$ which completes the proof of (ii). The properties (iii)-(v) follow easily from (ii) and the basic properties of local spectra and spectral subspaces.
	\end{proof}

\subsection*{A final remark: stability of the Dunford property}

Finally, in the class of linear fractional composition operators, we address the question whether the product of two commuting operators with the Dunford property $(C)$ inherits this property (see \cite{AG} for recent results). It is well-known that if two linear fractional composition operators  $C_{\varphi}$ and $C_{\psi}$ commute, $\varphi$ and $\psi$ has the same fixed points (see \cite{CDMV}, for instance). In addition, $C_{\varphi}C_{\psi}$ is the composition operator $C_{\psi \circ \varphi}$ and each of the classes in the classification of the linear transformations of $\D$ according to their fixed points remains invariant under composition. Hence, keeping in mind the Table \ref{tabla}, we are left with two cases:
if $\varphi$ is a parabolic automorphism and $\psi$ a parabolic non-automorphism (or viceversa) and if $\varphi$ is an elliptic automorphism an $\psi$ a loxodromic map with fix points 0 and $\infty$ (or viceversa). In the first case, $\psi \circ \varphi$ is a parabolic non-automorphism and in the second one is a loxodromic map. Accordingly, in both cases  $C_{\psi \circ \varphi}$ is decomposable, and therefore satisfies the Dunford property $(C)$.

\end{document}